\RequirePackage{fix-cm}
\documentclass[smallextended]{svjour3}       % onecolumn (second format)
\smartqed  % flush right qed marks, e.g. at end of proof
\usepackage{graphicx}
\begin{document}

\title{ Cubic Derivations on Banach
Algebras}

\titlerunning{ Cubic Derivations on Banach
Algebras}        % if too long for running head

\author{Abasalt Bodaghi   
}

\authorrunning{Abasalt Bodaghi} % if too long for running head

\institute{Abasalt Bodaghi \at
              Department of Mathematics, Garmsar Branch, Islamic Azad
University, Garmsar, Iran \\
              Tel.: +98-232-4225010\\
                            \email{abasalt.bodaghi@gmail.com} }          %  \\

\date{Accepted in Acta Mathematica Vietnamica}

\maketitle

\begin{abstract}
Let $A$ be a Banach algebra and $X$ be a Banach $A$-bimodule. A
mapping $D :A\longrightarrow X$ is a cubic derivation if $D$ is a
cubic homogeneous mapping, that is $D$ is cubic and $D(\lambda
a)={\lambda}^3 D(a)$ for any complex number $\lambda$ and all $a\in A$, and
$D(ab)=D(a)\cdot b^3 +a^3\cdot D(b)$ for all $a,b\in A$. In this
paper, we prove the stability of a cubic derivation with direct
method. We also employ a fixed point method to establish of the
stability and the superstability for cubic derivations.
\end{abstract}

\section{Introduction}
\label{intro}
In 1940, Ulam \cite{ul} posed the following question concerning
the stability of group homomorphisms: Under what condition does
there is an additive mapping near an approximately additive
mapping between a group and a metric group? One year later, Hyers
 \cite{h} answered the problem of Ulam under the assumption that the
groups are Banach spaces. This problem for linear mapping on Banach spaces was solved by J. M. Rassias in \cite{ra1}. A generalized version of the theorem of
Hyers for approximately linear mappings was given by Th. M.
Rassias \cite{ra}. Subsequently, the stability problems of various
functional equation have been extensively investigated by a
number of authors (for example, \cite{bag}, \cite{jp} and \cite{ph}). In
particular, one of the functional equations which has been
studied frequently is the cubic functional equation:
\begin{eqnarray}\label{a0} f(2x+y)+f(2x-y)=2f(x+y)+2f(x-y)+12f(x)
\end{eqnarray} 

The cubic function $f(x)=ax^3$ is a solution of this functional equation. The stability of
the functional equation (\ref{a0}) has been considered on
different spaces by a number of writers (for instance, \cite{n} and
\cite{ravi}).

In 2003, C$\breve{a}$dariu and Radu applied a fixed point method
to the investigation of the Jensen functional equation. They
presented a short and a simple proof for the Cauchy functional
equation  and the quadratic functional equation in \cite{cr2} and
\cite{cara}, respectively. After that, this method has been
applied by many authors to establish of miscellaneous functional
equations (see \cite{bo1}, \cite{ebp} and \cite{p7}).

In \cite{es}, Eshaghi Gordji et al. introduced the concept of a
cubic derivation which is a different notion of the current
paper. In fact, they did not consider the homogeneous property of
such derivations. In that paper,  the authors studied the
stability of cubic derivations on commutative Banach algebras.
The stability and the superstability of cubic double centralizers
and cubic multipliers on Banach algebras has been earlier proved
in \cite{lee}.

In this paper, we prove the stability of cubic derivations on
Banach algebras. An example of such derivations is indicated as
well. Using a fixed point
theorem, we also show that a cubic derivation can be superstable.

\section{Stability of cubic derivations}
\label{sec:1}
Let $A$ be a Banach algebra. A Banach space $X$ which is also a
left $A$-module is said to be a {\it left Banach $A$-module}
 if there is $k> 0$ such that
$$\|a\cdot x\| \leq k \|a\| \|x\|.$$

Similarly, a right Banach $A$-module and a Banach $A$-bimodule
are defined. 
Throughout this paper, we assume that $A$ is a Banach algebra,
$X$ is a Banach $A$-bimodule and denote $\overbrace{A\times
A\times...\times A}^{n-times}$ by $A^n$. For a natural number $n_0$, we
define $\bf T_{\frac{1}{n_0}} :=\{e^{i\theta} \ ;\
0\leq\theta\leq\frac{2\pi}{n_0} \}$ and denote $\bf
T_{\frac{1}{n_0}}$ by $\bf T$ when $n_0=1$. We also denote the set of all positive integers numbers, real numbers and complex numbers by $\bf N$, $\bf R$ and $\bf C$, respectively. First let us show by a example that the cubic derivations exist
on Banach algebras. Indeed, the following example is taken from
\cite{es} with the non-trivial module actions while in the
mentioned paper the left module action is zero.

\paragraph{\bf Example} Let $A$ be a Banach algebra.
Consider
\[{\mathcal T}:= \left[ \begin{array}{cccc}
{0} & {A} & {A} & {A} \\
{0} & {0} & {A} & {A}\\
{0} & {0} & {0} & {A}\\
{0} & {0} & {0} & {0} \\
 \end{array} \right]. \]
 
Then  $\mathcal T$ is a Banach algebra with the sum and product
being given by the usual $4\times 4$ matrix operations and with
the following norm:
\[ \|\left[ \begin{array}{cccc}
{0} & {a_1} & {a_2} & {a_3}\\
{0} & {0} & {a_4} & {a_5} \\
{0} & {0} & {0} & {a_6}\\
{0} & {0} & {0} & {0}\\
 \end{array} \right]\|=\sum_{j=1}^6\|a_j\| \hspace {0.7 cm} (a_j\in A). \]
 
So
\[{\mathcal T^*} = \left[ \begin{array}{cccc}
{0} & {A}^* & {A}^* & {A}^* \\
{0} & {0} & {A}^* & {A}^*\\
{0} & {0} & {0} & {A}^*\\
{0} & {0} & {0} & {0} \\
 \end{array} \right], \]
is the dual of $\mathcal T$ equipped with the following norm:
\[ \|\left[ \begin{array}{cccc}
{0} & {f_1} & {f_2} & {f_3}\\
{0} & {0} & {f_4} & {f_5} \\
{0} & {0} & {0} & {f_6}\\
{0} & {0} & {0} & {0}\\
 \end{array} \right]\|=Max\{\| f_j \| : 0\leq j\leq 6\} \hspace {0.7 cm} (f_j\in A^*). \]
 
Suppose that $\mathcal A= \left[ \begin{array}{cccc}
{0} & {a_1} & {a_2} & {a_3}\\
{0} & {0} & {a_4} & {a_5} \\
{0} & {0} & {0} & {a_6}\\
{0} & {0} & {0} & {0}\\
 \end{array} \right], \mathcal X= \left[ \begin{array}{cccc}
{0} & {x_1} & {x_2} & {x_3}\\
{0} & {0} & {x_4} & {x_5} \\
{0} & {0} & {0} & {x_6}\\
{0} & {0} & {0} & {0}\\
 \end{array} \right]\in \mathcal T$ and $\mathcal F=\left[ \begin{array}{cccc}
{0} & {f_1} & {f_2} & {f_3}\\
{0} & {0} & {f_4} & {f_5} \\
{0} & {0} & {0} & {f_6}\\
{0} & {0} & {0} & {0}\\
 \end{array} \right]\in \mathcal T^*$ in which $f_j\in A^*$ and $a_j,x_j\in A\,\, (0\leq j\leq 6)$. Consider the module actions of $\mathcal T$ on $\mathcal T^*$ as
 follows:
\[ \langle \mathcal F\cdot\mathcal A, \mathcal X
\rangle=\sum_{j=1}^6f(a_jx_j),\quad \langle \mathcal
A\cdot\mathcal F, \mathcal X \rangle=\sum_{j=1}^6f(x_ja_j).\]

Then $\mathcal T^*$  is a Banach $\mathcal T$-bimodule. Let  $\mathcal
G_0= \left[\begin{array}{cccc}
{0} & {g_1} & {g_2} & {g_3}\\
{0} & {0} & {g_4} & {g_5} \\
{0} & {0} & {0} & {g_6}\\
{0} & {0} & {0} & {0}\\
 \end{array} \right]\in \mathcal T^*$. Define $D:\mathcal
 T\longrightarrow \mathcal
 T^*$ via \[D(\mathcal A)=\mathcal
G_0\cdot\mathcal A^3-\mathcal A^3\cdot\mathcal G_0 \hspace {0.5
cm}(\mathcal A \in \mathcal T).
 \]\\
 
Given $\mathcal A= \left[ \begin{array}{cccc}
{0} & {a_1} & {a_2} & {a_3}\\
{0} & {0} & {a_4} & {a_5} \\
{0} & {0} & {0} & {a_6}\\
{0} & {0} & {0} & {0}\\
 \end{array} \right], \mathcal B= \left[ \begin{array}{cccc}
{0} & {b_1} & {b_2} & {b_3}\\
{0} & {0} & {b_4} & {b_5} \\
{0} & {0} & {0} & {b_6}\\
{0} & {0} & {0} & {0}\\
 \end{array} \right]\in \mathcal T$, we have

\begin{eqnarray}\label{r1}
\nonumber \langle D(2\mathcal A+\mathcal B), \mathcal
X\rangle&=&\langle \mathcal G_0\cdot(2\mathcal A+\mathcal
B)^3-(2\mathcal A+\mathcal B)^3\cdot\mathcal G_0,
\mathcal X\rangle \\
\nonumber &=&g_3((2a_1+b_1)(2a_4+b_4)(2a_6+b_6)x_3)\\
&-&g_3(x_3(2a_1+b_1)(2a_4+b_4)(2a_6+b_6)).
\end{eqnarray}

Similarly,
\begin{eqnarray}\label{r2}\nonumber\langle D(2\mathcal
A-\mathcal B), \mathcal
X\rangle &=&g_3((2a_1-b_1)(2a_4-b_4)(2a_6-b_6)x_3)\\
&-&g_3(x_3(2a_1-b_1)(2a_4-b_4)(2a_6-b_6)).
\end{eqnarray} 

On the other hand,
\begin{eqnarray}\label{r3}
\nonumber \langle 2D(\mathcal A+\mathcal B), \mathcal
X\rangle&=&\langle 2\mathcal G_0\cdot(\mathcal A+\mathcal
B)^3-2(\mathcal A+\mathcal B)^3\cdot\mathcal G_0,
\mathcal X\rangle \\
\nonumber &=&2g_3((a_1+b_1)(a_4+b_4)(a_6+b_6)x_3)\\
&-&2g_3(x_3(a_1+b_1)(a_4+b_4)(a_6+b_6)),
\end{eqnarray}
and
\begin{eqnarray}\label{r4}\nonumber\langle 2D(\mathcal
A-\mathcal B), \mathcal
X\rangle&=&\langle 2\mathcal G_0\cdot(\mathcal A-\mathcal
B)^3-2(\mathcal A-\mathcal B)^3\cdot\mathcal G_0,
\mathcal X\rangle \\
\nonumber &=&2g_3((a_1-b_1)(a_4-b_4)(a_6-b_6)x_3)\\
&-&2g_3(x_3(a_1-b_1)(a_4-b_4)(a_6-b_6)).\end{eqnarray} 

Also,
\begin{eqnarray}\label{r5}
   \nonumber \langle 12D(\mathcal A), \mathcal X\rangle &=& \langle
12\mathcal G_0\cdot\mathcal A^3-12\mathcal A^3\cdot\mathcal G_0,
\mathcal X\rangle  \\
      &=& 12g_3(a_1a_4a_6x_3)-12g_3(x_3a_1a_4a_6).
\end{eqnarray}

If follows from (\ref{r1})-(\ref{r5}) that
$$D(2\mathcal A+\mathcal B)+D(2\mathcal A-\mathcal B)=2D(\mathcal A+\mathcal B)+2D(\mathcal A-\mathcal B)+12D(\mathcal
A)$$ for all $\mathcal A, \mathcal B\in \mathcal T$. This shows
that $D$ is a cubic mapping. It is easy to see that $D(\lambda
\mathcal A)=\lambda^3 D(\mathcal A)$ for all $\mathcal A\in
\mathcal T$ and $\lambda\in \bf C$. Thus $D$ is a cubic
homogeneous mapping. Since $\mathcal T^4=\{0\}$, we have
$D(\mathcal A\mathcal B)=D(\mathcal A)\cdot\mathcal
B^3+\mathcal A^3\cdot D(\mathcal B)=0$ for all $\mathcal A,
\mathcal B \in \mathcal T$. Hence, $D$ is a cubic derivation.

 Now, we are going to prove the stability of cubic derivations on
Banach algebras.

\begin{theorem}\label{th1}
Suppose that $f: A \longrightarrow X$ is a mapping with $f(0) = 0$
for which there exists a function $\varphi  :A^4\longrightarrow [0,\infty)$
such that

\begin{eqnarray}\label{0}\widetilde{\varphi}(a,b,c,d):=\sum_{k=0}^{\infty}
\frac{1}{8^{k}}\varphi(2^ka,2^kb,2^kc,2^kd)<\infty\end{eqnarray}

\begin{eqnarray*}\label{1}
\|f(2\lambda a+\lambda b)+f(2\lambda a-\lambda b)-2\lambda^3
f(a+b)-2\lambda^3
f(a-b)-12\lambda^3f(a)\|\end{eqnarray*}\begin{eqnarray}\label{a1}\leq\varphi(a,b,0,0)
\end{eqnarray}
\begin{eqnarray}\label{2}
\|f(cd)-f(c)\cdot d^3-c^3\cdot f(d)\|\leq\varphi(0,0,c,d)
\end{eqnarray}
for all $\lambda\in \bf T_{\frac{1}{n_0}}$ and  all $a,b,c,d \in
A.$ Also, if for each fixed $a\in A$ the mappings $t\mapsto
f(ta)$ from $\bf R$ to $X$ is continuous, then there exists
a unique cubic derivation $D:A\longrightarrow X$ satisfying
\begin{eqnarray}\label{3}
\|f(a)-D(a)\|\leq \frac{1}{16}\widetilde{\varphi}(a,0,0,0)
\end{eqnarray}
for all  $a\in A$.
\end{theorem}

\begin{proof} Putting $b=0$ and $\lambda=1$ in (\ref{1}) , we have
\begin{eqnarray}\label{4}
\|\frac{1}{8}f(2a)-f(a)\|\leq\frac{1}{16}\varphi(a,0,0,0)
\end{eqnarray}
for all $a\in A$ (indeed $1\in \bf T_{\frac{1}{n}}$ for all
$n\in \bf N$). We replace $a$ by $2a$ in (\ref{4}) and
continue this method to get
\begin{eqnarray}\label{4.5}\left\|\frac{f(2^na)}{8^n}-f(a)\right\|\leq\frac{1}{16}\sum_{k=0}^{n-1}
\frac{\varphi(2^ka,0,0,0)}{8^k}\end{eqnarray} 

On the other hand, we can use induction to obtain
\begin{eqnarray}\label{5}\left\|\frac{f(2^na)}{8^n}-\frac{f(2^ma)}{8^m}\right\|\leq\frac{1}{16}\sum_{k=m}^{n-1}
\frac{\varphi(2^ka,0,0,0)}{8^k}\end{eqnarray} for all $a\in A$,
and $n>m\geq 0$. It follows from (\ref{0}) and (\ref{5}) that
sequence $\left\{\frac{f(2^na)}{8^n}\right\}$ is Cauchy. Since
$A$ is complete, this sequence convergence to the map $D$, that is
\begin{eqnarray}\label{6}
\lim_{n\to\infty}\frac{f(2^na)}{8^n}=D(a)
\end{eqnarray}

Taking the limit as $n$ tend to infinity in (\ref{4.5}) and
applying (\ref{6}), we can see that the inequality (\ref{3})
holds. Now, replacing $a, b$ by $2^na, 2^nb$, respectively in
(\ref{1}), we get
$$\|\frac{f(2^{n}(2\lambda a+b))}{8^{n}}-\frac{f(2^{n}(2\lambda
a-b))}{8^{n}}-2\lambda^{3}\frac{f(2^{n}(a+b))}{8^{n}}$$ $$
-2\lambda^{3}\frac{f(2^{n}(a-b))}{8^{n}}
-12\lambda^{3}\frac{f(2^{n}a)}{8^{n}}\|\leq\frac{\varphi(2^na,2^nb,0,0)}{8^n}.$$

Letting the limit as $n\longrightarrow \infty$, we obtain
\begin{eqnarray}\label{7} D(2\lambda a+\lambda b)+D(2\lambda a-\lambda b)=2\lambda^3 D(a+b)+2\lambda^3
D(a-b)+12\lambda^3f(a)\end{eqnarray}
 for all $a,b\in A$ and all
$\lambda\in \bf T_{\frac{1}{n_0}}$. If follows from (\ref{7})
that $D$ is a cubic mapping when $\lambda=1$. Letting $b=0$ in
(\ref{7}), we get $D(\lambda a)=\lambda^3 D(a)$ for all $a\in A$
and $\lambda\in \bf T_{\frac{1}{n_0}}$. Now, let
$\lambda=e^{i\theta}\in \bf{T}$. We set $\lambda_0=e^{\frac{i
\theta }{n_0}}$, thus $\lambda_0$ belongs to
$\bf{T}_\frac{1}{n_0}$ and $D(\lambda
a)=D(\lambda_0^{n_0}a)=\lambda_0^{3n_0}D(a)=\lambda^3 D(a)$ for
all $a\in A$. Under the assumption that $f(ta)$ is continuous in
$t\in \bf{R}$ for each fixed $a\in A$, by the same reasoning
as in the proof of \cite{Cz},  $D(\lambda a)=\lambda^3D(a)$ for
all $\lambda\in \bf{R}$ and $a\in A$. So,
$$D(\lambda a)=D(\frac{\lambda}{|\lambda|}~~|\lambda|
a)=\frac{\lambda^3}{|\lambda|^3}D(|\lambda|a)=\frac{\lambda^3}{|\lambda|^3}|\lambda|^3D(a)=\lambda^3D(a),$$
for all $a\in A$ and $\lambda\in \bf C\,\, (\lambda\neq 0)$.
Therefore, $D$ is cubic homogeneous mapping.
 If we replace $c,d$ by $2^nc, 2^nd$
respectively in (\ref{2}), we have
$$\|\frac{f(2^{2n}cd)}{8^{2n}}-\frac{f(2^{n}c)}{8^{n}} \cdot d^3-c^3\cdot\frac{f(2^{n}d)}{8^{n}}\|\leq\frac{\varphi(0,0,2^nc,2^nd)}{8^{2n}}\leq\frac{\varphi(0,0,2^nc,2^nd)}{8^n}.$$
for all $c,d\in A$. Taking the limit as $n\longrightarrow\infty$,
we get $D(cd)=D(c)\cdot d^3+c^3\cdot D(d)$, for all $c,d\in A$.
This shows that $D$ is a cubic derivation.

Now, let $D':A\longrightarrow X$ be another cubic derivation satisfying
(\ref{3}). Then we have

\begin{eqnarray*}
   \|D(a)-D'(a)\| &=& \frac{1}{8^{n}}\|D(2^{n}a)-D'(2^{n}a)\| \\
      &\leq& \frac{1}{8^{n}}(\|D(2^{n}a)-f(2^{n}a)\|+\|f(2^{n}a)-D'(2^{n}a)\|)\\
         &\leq&  \frac{1}{8^{n+1}}\widetilde{\varphi}(2^na,0,0,0)\\
      &=& \frac{1}{8}\sum_{k=0}^{\infty}
\frac{1}{8^{n+k}}\varphi(2^{n+k}a,0,0,0)\\
      &=& \frac{1}{8}\sum_{k=n}^{\infty}
\frac{1}{8^{k}}\varphi(2^{k}a,0,0,0)
\end{eqnarray*}
for all $a\in A$. By letting $n\longrightarrow\infty$ in the
preceding inequality, we immediately find the uniqueness of $D$.
This completes the proof.
\end{proof}
\begin{corollary}\label{coo}Let $\delta, r$ be positive real numbers with $r<3$, and
let $f: A \longrightarrow X$ be a mapping with $f(0) = 0$ such
that
\begin{eqnarray*}\label{9}
\|f(2\lambda a+\lambda b)+f(2\lambda a-\lambda b)-2\lambda^3
f(a+b)-2\lambda^3f(a-b)-12\lambda^3f(a)\|\end{eqnarray*}\begin{eqnarray}\label{a1}\leq\delta
(\|a\|^r + \|b\|^r)
\end{eqnarray}
\begin{eqnarray}\label{10}
\|f(cd)-f(c)\cdot d^3-c^3\cdot f(d)\|\leq\delta(\|c\|^r+\|d\|^r)
\end{eqnarray}
for all $\lambda\in \bf T_{\frac{1}{n_0}}$ and  all $a,b,c,d \in
A.$ Then there exists a unique cubic derivation $D: A
\longrightarrow X$ satisfying
\begin{eqnarray}\label{11}
\|f(a)-D(a)\|\leq \frac{\delta}{2(8-2^{r})}\|a\|^r
\end{eqnarray}
for all $a\in A$.
\end{corollary}

\begin{proof}
It follows from Theorem \ref{th1} by taking
\begin{eqnarray*} \varphi(a,b,c,d)=\delta (\|a\|^r + \|b\|^r+ \|c\|^r +\|d\|^r).\end{eqnarray*}
\end{proof}

\begin{theorem}\label{th2}
Suppose that $f: A \longrightarrow X$ is a mapping with $f(0) = 0$
for which there exists a function $\varphi  :A^4\longrightarrow [0,\infty)$
satisfying \emph{(}\ref{1}\emph{)}, \emph{(}\ref{2}\emph{)} and
$$\widetilde{\varphi}(a,b,c,d):=\sum_{k=1}^{\infty}
8^{k}\varphi(2^{-k}a,2^{-k}b,2^{-k}c,2^{-k}d)<\infty,$$ for all
$a,b,c,d \in A.$ Also, if for each fixed $a\in A$ the mappings
$t\mapsto f(ta)$ from $\bf{R}$ to $A$ is continuous, then
there exists a unique cubic derivation $D: A \longrightarrow X$
satisfying
\begin{eqnarray}\label{20}
\|f(a)-D(a)\|\leq \frac{1}{16}\widetilde{\varphi}(a,0,0,0)
\end{eqnarray}
for all $a\in A$.
\end{theorem}

\begin{proof} Putting $b=0$ and $\lambda=1$ in (\ref{1}), we have
\begin{eqnarray}\label{21}
\|f(2a)-8f(a)\|\leq\frac{1}{2}\varphi(a,0,0,0)
\end{eqnarray}
for all $a\in A$. We replace $a$ by $\frac{a}{2}$ in (\ref{21}) to
obtain
\begin{eqnarray}\label{22}
\|f(a)-8f(\frac{a}{2})\|\leq\frac{1}{2}\varphi(\frac{a}{2},0,0,0)
\end{eqnarray}

Using triangular inequality and proceeding this way, we have
\begin{eqnarray}\label{23}\left\|f(a)-8^nf(\frac{a}{2^n})\right\|\leq\frac{1}{16}\sum_{k=1}^{n}
8^k\varphi(\frac{a}{2^k},0,0,0)\end{eqnarray}

 If we show that the
sequence $\left\{8^nf(\frac{a}{2^n})\right\}$ is Cauchy, then it
will be convergent by the completeness of $A$. For this, replace
$a$ by $\frac{a}{2^m}$ in (\ref{23}) and then multiply both side
by $8^{m}$, we get
\begin{eqnarray*}
\left\|8^mf(\frac{a}{2^m})-8^{m+n}f(\frac{a}{2^{m+n}})\right\|&\leq&\frac{1}{16}\sum_{k=1}^{n}
8^{k+m}\varphi(\frac{a}{2^{k+m}},0,0,0)\\
&=&\frac{1}{16}\sum_{k=m+1}^{m+n}
8^{k}\varphi(\frac{a}{2^{k}},0,0,0)
\end{eqnarray*}
 for all $a\in
A$, and $n>m\geq 0$. Thus the mentioned sequence is convergent  to
the map $D$, i.e.,
\begin{eqnarray*}
D(a)=\lim_{n\to\infty}8^nf(\frac{a}{2^n}).
\end{eqnarray*}

Now, similar to the proof of Theorem \ref{th1}, we can continue
the rest of the proof.
\end{proof}

\begin{corollary}\label{coco}Let $\delta, r$ be positive real numbers  with $r>3$, and
let $f: A \longrightarrow X$ be a mapping with $f(0) = 0$
satisfying \emph{(}\ref{9}\emph{)}, \emph{(}\ref{10}\emph{)}. Then there exists a unique
cubic derivation $D: A \longrightarrow X$ satisfying
\begin{eqnarray}\label{13}
\|f(a)-D(a)\|\leq \frac{\delta}{2(2^{r}-8)}\|a\|^r
\end{eqnarray}
for all $a\in A$.
\end{corollary}

\begin{proof} The result follows from Theorem  \ref{th2} by putting
\begin{eqnarray*} \varphi(a,b,c,d)=\delta (\|a\|^r + \|b\|^r+ \|c\|^r +\|d\|^r).\end{eqnarray*}
\end{proof}

%----------------------------------------------------------------------------------------------------------------------------------
%----------------------------------------------------------------------------------------------------------------------------------
%----------------------------------------------------------------------------------------------------------------------------------

\section{A fixed point approach}
In this section, we prove the
stability and the superstability for cubic derivations on Banach algebras by using a fixed point theorem. First, we bring the following fixed point theorem which is proved
in \cite{dia}. This theorem plays a fundamental role to achieve
our purpose in this section (an extension of the result was given
in \cite{tur}).

\begin{theorem}\label{t} \emph{(}The fixed point alternative\emph{)} Let $(\Delta,d)$ be a complete generalized metric space
 and $\mathcal J:\Delta\longrightarrow \Delta$ be a mapping
 with a Lipschitz constant $L<1$. Then, for each element $\alpha\in \Delta$, either $d(\mathcal J^n \alpha, \mathcal J^{n+1}
\alpha)=\infty~$ for all $n\geq0,$ or there exists a natural
number $n_{0}$ such that:
\begin{enumerate}
\item[\emph{(i)}] {$d(\mathcal J^n \alpha, \mathcal J^{n+1} \alpha)<\infty ~$for all $n \geq n_{0}$;}
\item[\emph{(ii)}] {the sequence $\{\mathcal J^n \alpha\}$ is convergent to a fixed
point $\beta^*$ of $\mathcal J$;}
\item[\emph{(iii)}] { $\beta^*$ is the unique fixed point of $\mathcal J$ in the set
$~\Delta_1=\{\beta\in \Delta:d(\mathcal T^{n_{0}} \alpha,
\beta)<\infty\}$;}
\item[\emph{(iv)}] {$d(\beta,\beta^*)\leq\frac{1}{1-L}d(\beta, \mathcal J\beta)$ for
all $~\beta\in\Delta_1$.}
\end{enumerate}
\end{theorem}

\begin{theorem}\label{th4}
Let $f :A \longrightarrow X$  be a continuous mapping with $f(0)=0$
and let $\phi  :A^2\longrightarrow [0,\infty)$ be a continuous function such
that
\begin{eqnarray*}
\|f(2\lambda a+\lambda b)+f(2\lambda a-\lambda b)-2\lambda^3
f(a+b)-2\lambda^3
f(a-b)-12\lambda^3f(a)\|\end{eqnarray*}\begin{eqnarray}\label{a1}\leq\phi(a,b)
\end{eqnarray}
\begin{eqnarray}\label{a2}
\|f(ab)-f(a)\cdot b^3-a^3\cdot f(b)\|\leq\phi(a,b)
\end{eqnarray}
for all $\lambda\in \bf T_{\frac{1}{n_0}}$ and  all $a,b \in A.$
If there exists a constant $k\in(0,1)$, such that
\begin{eqnarray}\label{a3}
\phi(2a,2b)\leq 8k \phi(a,b)
\end{eqnarray}
for all $a,b \in A$, then there exists a unique cubic derivation
$D:A \longrightarrow X$ satisfying
\begin{eqnarray}\label{a4}
\|f(a)-D(a)\|\leq \frac{1}{16(1-k)}\phi(a,0)
\end{eqnarray}
for all $a\in A$.
\end{theorem}

\begin{proof}To achieve our goal, we make the conditions of Theorem \ref{t}. We
consider the set
$$\Delta=\{g:A\longrightarrow X|\,\, g(0)=0\}$$
and define the mapping $d$ on $\Delta\times \Delta$ as follows:
$$d(g,h):=inf\{c\in (0,\infty):\|g(a)-h(a)\|\leq c\phi(a,0),\quad (\forall a\in A)\},$$
if there exist such constant $c$, and $d(g,h)=\infty$, otherwise.
Similar to the proof of \cite[Theorem 2.2]{bag}, we can show that
$d$ is a generalized metric on $\Delta$ and the metric space
$(\Delta,d)$ is complete. Now, we define the mapping $\mathcal
J:\Delta\longrightarrow \Delta$ by
\begin{eqnarray}\label{a5}
\mathcal Jh(a)=\frac{1}{8}h(2a),\quad(a\in A).\end{eqnarray}

If $g,h\in \Delta$ such that $d(g,h)<c$, by definition of $d$ and
$\mathcal J$, we have
$$
\left\|\frac{1}{8}g(2a)-\frac{1}{8}h(2a)\right\|\leq\frac{1}{8}c\phi(2a,0)
$$
for all $a\in A$. Applying (\ref{a3}), we get
$$
\left\|\frac{1}{8}g(2a)-\frac{1}{8}h(2a)\right\|\leq ck\phi(a,0)
$$
for all $a\in A$. The above inequality shows that $d(\mathcal
Jg,\mathcal Jh)\leq k d(g,h)$ for all $g,h\in \Delta$. Hence,
$\mathcal J$ is a strictly contractive mapping on $\Delta$ with a
Lipschitz constant $k$. Here, we prove that $d(\mathcal Jf,f)<\infty$. Putting $b=0$ and
$\lambda=1$ in (\ref{a1}), we obtain
$$\|2f(2a)-16f(a)\|\leq\phi(a,0)$$
for all $a\in A$. Hence
\begin{eqnarray}\label{a6}
\|\frac{1}{8}f(2a)-f(a)\|\leq\frac{1}{16}\phi(a,0)
\end{eqnarray}
for all $a\in A$. We conclude from (\ref{a6}) that $d(\mathcal J
f,f)\leq\frac{1}{16}$. It follows from Theorem \ref{t} that
$d(\mathcal J^n g, \mathcal J^{n+1} g)<\infty ~$ for all $n \geq
0$, and thus in this Theorem we have $n_0=0$. Therefore the parts
(iii) and (iv) of Theorem \ref{t} hold on the whole $\Delta$.
Hence there exists a unique mapping $D: A\longrightarrow X$ such that
$D$ is a fixed point of $\mathcal J$ and that $\mathcal J
^nf\rightarrow D$ as $n \rightarrow \infty$. Thus
\begin{eqnarray}\label{a7}
\lim_{n\to\infty}\frac{f(2^na)}{8^n}=D(a)
\end{eqnarray}
for all $a\in A$, and so
$$d(f,D)\leq\frac{1}{1-k}d(\mathcal J f,f)\leq\frac{1}{16(1-k)}.$$

The above equalities show that (\ref{a4}) is true for all $a\in
A$. Now, it follows from (\ref{a3}) that
\begin{eqnarray}\label{a8}
\lim_{n\to\infty}\frac{\phi(2^na,2^nb)}{8^n}=0.
\end{eqnarray}

Replacing $a$ and $b$ by $2^na$ and $2^nb$ respectively in
(\ref{a1}), we get
$$\|\frac{f(2^{n}(2\lambda a+b))}{8^{n}}-\frac{f(2^{n}(2\lambda
a-b))}{8^{n}}-2\lambda^{3}\frac{f(2^{n}(a+b))}{8^{n}}$$ $$
-2\lambda^{3}\frac{f(2^{n}(a-b))}{8^{n}}
-12\lambda^{3}\frac{f(2^{n}a)}{8^{n}}\|\leq\frac{\phi(2^na,2^nb)}{8^n}.$$

Taking the limit as $n\longrightarrow \infty$, we obtain
\begin{eqnarray}\label{a8} D(2\lambda a+\lambda b)+D(2\lambda a-\lambda b)=2\lambda^3 D(a+b)+2\lambda^3
D(a-b)+12\lambda^3D(a)\end{eqnarray}
 for all $a,b\in A$ and all
$\lambda\in \bf T_{\frac{1}{n_0}}$. By (\ref{a8}), $D$ is a cubic
mapping when $\lambda=1$. Letting $b=0$ in (\ref{a8}), we get
$D(\lambda a)=\lambda^3 D(a)$ for all $a\in A$ and $\lambda\in
\bf T_{\frac{1}{n_0}}$. Similar to the proof of Theorem
\ref{th1}, we have $D(\lambda a)=\lambda^3 D(a)$ for all $a\in A$
and $\lambda\in \bf T$. Since $D$ is a cubic mapping,
$D(ra)=r^3D(a)$ for any rational number $r$. It follows from the
continuity of $f$ and $\phi$ that for each $\lambda\in
\bf{R}$, $D(\lambda a)=\lambda^3D(a)$.  The proof of Theorem
\ref{th1} indicates that $D(\lambda a)=\lambda^3D(a),$ for all
$a\in A$ and $\lambda\in \bf C \,\, (\lambda\neq 0)$. Therefore,
$D$ is a cubic homogeneous map. If we replace $a,b$ by $2^na,
2^nb$ respectively in (\ref{a2}), we have
$$\|\frac{f(2^{2n}ab)}{8^{2n}}-\frac{f(2^{n}a)}{8^{n}}\cdot b^3-a^3\cdot\frac{f(2^{n}b)}{8^{n}}\|\leq\frac{\phi(2^na,2^nb)}{8^{2n}}\leq\frac{\phi(2^na,2^nb)}{8^n}.$$
for all $a,b\in A$. Taking the limit as $n$ tend to infinity, we
get $D(ab)=D(a)\cdot b^3+a^3\cdot D(b)$, for all $a,b\in A$.
Therefore $D$ is a unique cubic derivation.
\end{proof}
In the following result, we get again Corollary \ref{coo} which is a direct consequence of the above Theorem.

\begin{corollary}\label{co1}
 Let $p, \delta$ be the nonnegative real numbers with $p<3$ and let $f :A \Longrightarrow X$  be
a mapping with $f(0)=0$ such that
\begin{eqnarray*}
\|f(2\lambda a+\lambda b)+f(2\lambda a-\lambda b)-2\lambda^3
f(a+b)-2\lambda^3 f(a-b)-12\lambda^3f(a)\|\end{eqnarray*}
\begin{eqnarray}\leq\delta
(\|a\|^p + \|b\|^p)
\end{eqnarray}
\begin{eqnarray}
\|f(ab)-f(a)\cdot b^3-a^3\cdot f(b)\|\leq\delta( \|a\|^p +\|b\|^p)
\end{eqnarray}
for all $\lambda\in \bf T_{\frac{1}{n_0}}$ and  all $a,b\in A.$
Then there exists a unique cubic derivation $D:A \longrightarrow X$
satisfying
\begin{eqnarray*}
\|f(a)-D(a)\|\leq \frac{\delta}{2(8-2^{p})}\|a\|^p
\end{eqnarray*}
for all $a\in A$.
\end{corollary}

\begin{proof} If we put $\phi(a,b)=\delta (\|a\|^p + \|b\|^p)$ in Theorem \ref{th4}, we obtain the desired result.
\end{proof}

In the next result, we show that under which conditions cubic derivations are superstable.

\begin{corollary}\label{cor3}
Let $p,q, \delta$ be non-negative real numbers with $0<p+q<3$ and
let $f :A \longrightarrow X$ be a mapping with $f(0)=0$ such that

\begin{eqnarray*}
\|f(2\lambda a+\lambda b)+f(2\lambda a-\lambda b)-2\lambda^3
f(a+b)-2\lambda^3 f(a-b)-12\lambda^3f(a)\|\end{eqnarray*}
\begin{eqnarray}\label{a21}\leq\delta(\|a\|^{p}\|b\|^{q})
\end{eqnarray}
\begin{eqnarray}\label{a22}
\|f(ab)-f(a)\cdot b^3-a^3\cdot
f(b)\|\leq\delta(\|a\|^{p}\|b\|^{q})
\end{eqnarray}
for all $\lambda\in \bf T_{\frac{1}{n_0}}$ and  all $a,b \in A$.
Then $f$ is a cubic derivation on $A$.
\end{corollary}

\begin{proof}Putting $a=b=0$ in (\ref{a21}), we get $f(0)=0$. Now, if we put
$b=0$, $\lambda=1$ in (\ref{a21}), then we have $f(2a)=8f(a)$ for
all $a\in A$. It is easy to see by induction that
$f(2^na)=8^{n}f(a)$, and so $f(a)=\frac{f(2^na)}{8^{n}}$ for all
$a\in A$ and $n\in\bf{N}$. It follows from Theorem \ref{th4}
that $f$ is a cubic homogeneous mapping. Letting
$\phi(a,b)=\delta(\|a\|^{p}\|b\|^{q})$ in Theorem \ref{th4}, we
can obtain the desired result.\end{proof}

Note that if a mapping $f:A \longrightarrow X$ satisfies the
inequalities (\ref{a21}) and (\ref{a22}), where $p,q$ and $\delta$
are non-negative real numbers such that $p+q>3$ and $p$ is greater
than $0$, then it is obvious that $f$ is a cubic derivation on
$A$ by putting $\phi(a,b)=\delta \left(\|a\|^{p}\|b\|^{q}\right)$
in Theorem \ref{th4}.

\begin{acknowledgements}
The author sincerely thank the anonymous reviewer for a careful reading,
constructive comments and fruitful suggestions to improve the quality
of the paper and suggesting a related reference.
\end{acknowledgements}



\end{document}